\documentclass{article}
\usepackage[latin1]{inputenc}
\usepackage[T1]{fontenc}
\usepackage{amsmath,amssymb}
\usepackage{amsfonts,amsthm}
\usepackage{bbm}

\usepackage{geometry}
\usepackage{color}
   \definecolor{cites}{rgb}{0.75 , 0.00 , 0.00}  
   \definecolor{urls} {rgb}{0.00 , 0.00 , 1.00}  
   \definecolor{links}{rgb}{0.00 , 0.00 , 0.5}   
  \definecolor{gray}{rgb}{0.5,0.5,.5}
\usepackage{graphicx}
\usepackage{cancel}
\usepackage[plainpages=false,colorlinks=true, citecolor=cites, urlcolor=urls, linkcolor=links, breaklinks=true]{hyperref}

\allowdisplaybreaks

\newcommand{\C}{\mathbb{C}}

\newcommand{\N}{\mathbb{N}}

\newcommand{\1}{\mathbbm{1}}

\newcommand{\Ac}{\mathcal{A}}
\newcommand{\Bc}{\mathcal{B}}
\newcommand{\Fc}{\mathcal{F}}
\newcommand{\Kc}{\mathcal{K}}
\newcommand{\Lc}{\mathcal{L}}
\newcommand{\Pc}{\mathcal{P}}

\newcommand{\af}{\mathfrak{a}}
\newcommand{\Af}{\mathfrak{A}}

\renewcommand{\epsilon}{\varepsilon}

\newcommand{\vertiii}[1]{{\left\vert\kern-0.25ex\left\vert\kern-0.25ex\left\vert #1
    \right\vert\kern-0.25ex\right\vert\kern-0.25ex\right\vert}}
\newcommand{\vertiiis}[1]{{\vert\kern-0.25ex\vert\kern-0.25ex\vert #1
    \vert\kern-0.25ex\vert\kern-0.25ex\vert}}

\DeclareMathOperator{\dist}{dist}
\DeclareMathOperator{\ess}{ess}

\DeclareMathOperator{\Osc}{Osc}

\DeclareMathOperator{\spec}{sp}

\DeclareMathOperator{\VMO}{VMO}
\DeclareMathOperator{\VO}{VO}

\DeclareMathOperator{\BDO}{BDO}

\DeclareMathOperator{\Real}{Re}
\renewcommand{\Re}{\Real}
\DeclareMathOperator{\Imag}{Im}
\renewcommand{\Im}{\Imag}

\newcommand{\from}{\colon}

\providecommand{\scpr}[2]{\left\langle #1, #2 \right\rangle}
\renewcommand{\sp}{\scpr}
\providecommand{\abs}[1]{\left\lvert#1\right\rvert}
\providecommand{\norm}[1]{\left\lVert#1\right\rVert}
\providecommand{\set}[1]{\left\{ #1\right\}}

\newtheorem{thm}{Theorem}

\newtheorem{lem}[thm]{Lemma}
\newtheorem{prop}[thm]{Proposition}
\newtheorem{cor}[thm]{Corollary}
\newtheorem*{cor*}{Corollary}

\theoremstyle{definition}
\newtheorem{defn}[thm]{Definition}
\newtheorem{ass}[thm]{Assumption}

\theoremstyle{remark}

\newtheorem{ex}[thm]{Example}
\newtheorem{qn}[thm]{Question}

\numberwithin{equation}{section}

\allowdisplaybreaks
\begin{document}
\title{\bf Toeplitz and related operators on polyanalytic Fock spaces}
\author{Raffael Hagger\footnote{Mathematisches Seminar, Christian-Albrechts-Universit\"at zu Kiel, Heinrich-Hecht-Platz 6, 24118 Kiel, Germany, hagger@math.uni-kiel.de}}
\maketitle

\begin{center}\textit{In Memory of Harold Widom}\end{center}

\begin{abstract}
We give a characterization of compact and Fredholm operators on polyanalytic Fock spaces in terms of limit operators. As an application we obtain a generalization of the Bauer--Isralowitz theorem using a matrix valued Berezin type transform. We then apply this theorem to Toeplitz and Hankel operators to obtain necessary and sufficient conditions for compactness. As it turns out, whether or not a Toeplitz or Hankel operator is compact does not depend on the polyanalytic order. For Hankel operators this even holds on the true polyanalytic Fock spaces.


\end{abstract}

\section{Introduction} \label{introduction}

Polyanalytic functions on $\C$ (also called polyentire functions) are smooth functions $f \from \C \to \C$ in the variables $z$ and $\bar z$ that satisfy
\[\frac{\partial^n f}{(\partial \bar z)^n} = 0\]
for some $n \in \N$. They can be represented in the form
\[f(z) = \sum\limits_{j = 0}^{n-1} h_j(z)\bar{z}^j,\]
where the $h_j$ are entire functions. We thrn say that $f$ is of polyanalytic order at most $n$. In terms of regularity, polyanalytic functions are somewhere in between (complex) analytic and real analytic functions. They still satisfy a Cauchy type integral equation and are subject to Liouville's theorem, but the maximum principle fails as well as the strong form of the identity theorem. For example, $f(z) = 1 - \abs{z}^2$ defines a polyanalytic function with a maximum at $0$, which vanishes on the unit circle; two features a non-zero analytic function cannot have. But of course we still have a weaker form of the identity theorem, which also holds for real analytic functions: A polyanalytic function that vanishes on an open set is equal to $0$ everywhere. We refer to \cite{Balk} for an overview of results on polyanalytic functions.

Polyanalytic functions have been studied for over a century as they naturally appear in the theory of elasticity \cite{Kolossov,Muskhelishvili}. However, it was only recently discovered that certain polyanalytic function spaces posess an interesting creation-annihilation structure similar to the quantum harmonic oscillator \cite{Vasilevski}. Subsequently, several connections to time-frequency analysis, signal processing and quantum mechanics have been found. We refer the interested reader to \cite{AbFei,RoVa} for now and return to related work after some introductory material.

Let $\mu$ denote the Gaussian measure on $\C$ defined by
\begin{equation} \label{eq:Gaussian}
\mathrm{d}\mu(z) = \frac{1}{\pi} e^{-\abs{z}^2} \, \mathrm{d}z.
\end{equation}
The polyanalytic Fock space $\Fc^2_n$ is the closed subset of $L^2(\C,\mu)$ consisting of polyanalytic functions of order at most $n$. For $n = 1$ we of course get the classical Fock space $\Fc^2 = \Fc^2_1$ of analytic functions. Vasilevski \cite{Vasilevski} now observed that the polyanalytic Fock spaces can be decomposed into an orthogonal sum of so-called true polyanalytic Fock spaces
\[\Fc^2_n = \bigoplus\limits_{k = 1}^n \Fc^2_{(k)},\]
where $\Fc^2_{(k)}$ consists of those $f \in L^2(\C,\mu)$ that can be written as
\[f(z) = \frac{1}{(k-1)!} e^{\abs{z}^2} \frac{\partial^{k-1}}{\partial z^{k-1}}\left(e^{-\abs{z}^2}g(z)\right)\]
for an entire function $g$. Now consider the following operators defined on $\Fc^2_n$:
\[\af^\dagger := \left(-\frac{\partial}{\partial z} + \bar{z}\right) \quad \text{and} \quad \af := \frac{\partial}{\partial \bar{z}}.\]
The operator $\frac{1}{\sqrt{k}}\af^\dagger$ is an isometric isomorphism between $\Fc^2_{(k)}$ and $\Fc^2_{(k+1)}$ with inverse $\frac{1}{\sqrt{k}}\af$. In particular, $N := \af^\dagger\af$ is the counting operator, that is,
\[Nf = kf \quad \text{for } f \in \Fc^2_{(k+1)},\]
and it holds $[\af,\af^\dagger] = I$. Summing up all the true polyanalytic Fock spaces, we obtain $L^2(\C,\mu)$:
\begin{equation} \label{eq:decomposition}
L^2(\C,\mu) = \bigoplus\limits_{k = 1}^\infty \Fc^2_{(k)};
\end{equation}
see \cite[Corollary 2.4]{Vasilevski}.

Just like $\Fc^2$, the true polyanalytic Fock spaces are reproducing kernel Hilbert spaces. Their reproducing kernels are given by
\[K_{(k)}(z,w) = \frac{1}{(k-1)!}\left(-\frac{\partial}{\partial \bar{w}} + w\right)^{k-1}\left(-\frac{\partial}{\partial z} + \bar{z}\right)^{k-1}e^{z\bar{w}} = L_{k-1}^0(\abs{z-w}^2)e^{z\bar{w}},\]
where for $\alpha \in \N_0$ the
\[L_k^\alpha(x) := \sum\limits_{j = 0}^k (-1)^j{k+\alpha \choose k-j}\frac{x^j}{j!}\]
are the generalized Laguerre polynomials. Consequently, the orthogonal projection onto $\Fc^2_{(k)}$ is given by
\[P_{(k)}f(z) = \int_{\C} f(w)L_{k-1}^0(\abs{z-w}^2)e^{z\bar{w}} \, \mathrm{d}\mu(w)\]
for $z \in \C$, $f \in L^2(\C,\mu)$. As $\Fc^2_n$ is equal to the orthogonal sum of true polyanalytic Fock spaces, the orthogonal projection $P_n$ onto $\Fc^2_n$ is just the sum of the $P_{(k)}$. Using an identity for Laguerre polynomials, we get
\[P_nf(z) = \int_{\C} f(w)L_{n-1}^1(\abs{z-w}^2)e^{z\bar{w}} \, \mathrm{d}\mu(w)\]
for $z \in \C$, $f \in L^2(\C,\mu)$.

Via the decomposition \eqref{eq:decomposition}, we can define the isometry
\[\Af^{\dagger} \from L^2(\C,\mu) \to L^2(\C,\mu), \quad \Af^{\dagger}f = \Af^{\dagger}\sum\limits_{k = 1}^{\infty} f_k := \sum\limits_{k = 1}^{\infty} \frac{1}{\sqrt{k}}\af^{\dagger}f_k,\]
where $f_k := P_{(k)}f \in \Fc_{(k)}^2$ is the $k$-th component of $f \in L^2(\C,\mu)$. Its adjoint is of course given by
\[\Af \from L^2(\C,\mu) \to L^2(\C,\mu), \quad \Af f = \Af\sum\limits_{k = 1}^{\infty} f_k = \sum\limits_{k = 2}^{\infty} \frac{1}{\sqrt{k-1}}\af f_k.\]
By definition, we have $\Af^{\dagger}(\Fc_{(k)}^2) = \Fc_{(k+1)}^2$, $\Af(\Fc_{(k+1)}^2) = \Fc_{(k)}^2$ and $\Af(\Fc_{(1)}^2) = \set{0}$. In particular, $\Af^{\dagger}$ and $\Af$ can be seen as the forward and backward shift on $\bigoplus\limits_{k = 1}^\infty \Fc^2_{(k)} \cong \ell^2(\N,\Fc^2)$. It is also clear that $\Af\Af^{\dagger} = I$ and $\Af^{\dagger}\Af = (I - P_{(1)})$.

For bounded functions $f$ we can now define polyanalytic Toeplitz and Hankel operators in the usual way:
\begin{gather*}
T_{f,(k)} \from \Fc^2_{(k)} \to \Fc^2_{(k)}, \quad T_{f,(k)}g = P_{(k)}(fg),\\
T_{f,n} \from \Fc^2_n \to \Fc^2_n, \quad T_{f,n}g = P_n(fg),\\
H_{f,(k)} \from \Fc^2_{(k)} \to L^2(\C,\mu), \quad H_{f,(k)}g = (I-P_{(k)})(fg),\\
H_{f,n} \from \Fc^2_n \to L^2(\C,\mu), \quad H_{f,n}g = (I-P_n)(fg).
\end{gather*}
For $k = 1$ (or equivalently $n = 1$) we just get the usual Toeplitz and Hankel operators on the standard analytical Fock space $\Fc^2 = \Fc^2_1 = \Fc^2_{(1)}$.

Shortly after Axler and Zheng \cite{AxZhe} proved a similar result for the Bergman space over the unit disk, Engli\v{s} \cite{Englis} showed that a Toeplitz operator on $\Fc^2$ is compact if and only if its Berezin transform vanishes at infinity. In fact, in both cases this is true not only for Toeplitz operators but for any finite sum of finite products of Toeplitz operators. This was later generalized to what we shall call the Bauer--Isralowitz theorem: A bounded linear operator on $\Fc^2$ is compact if and only if it is in the $C^*$-algebra generated by all Toeplitz operators and its Berezin transform vanishes at infinity. Motivated by this result, many allegedly larger $C^*$-algebras of operators, such as the sufficiently and weakly localized operators \cite{IsMiWi,XiaZhe}, have been introduced where the same result would hold. However, in 2015 Xia \cite{Xia} proved the surprising result that all these algebras actually coincide with the closure of the set of all Toeplitz operators. All the different approaches notably had some limit operator type arguments in common, reminiscent of the Fredholm theory of sequence spaces. Consequently, in analogy to the sequence space case, another $C^*$-algebra, called the band-dominated operators, was introduced in \cite{FuHa}, which formalized the limit operator idea. It was later shown in \cite{BaFu} that this would be again the same algebra, but the band-dominated approach of \cite{FuHa} also provided a  characterization of Fredholm operators. Moreover, it allowed to study Hankel operators, providing quick proofs for well-known compactness results as well as some new insights \cite{HaVi}.

Apart from different algebras, several authors also started considering different domains such as bounded symmetric domains \cite{Hagger_BSD} and Bergman-type function spaces \cite{MiWi} just to name a few. The structure of these results is always very similar. In order to be compact, an operator must be contained in a certain $C^*$-algebra and the Berezin transform must vanish at the boundary of the domain. In this paper we now present an example where the Berezin transform is not strong enough to characterize compactness, even within the reasonable algebra of band-dominated operators, which, as mentioned above, at least in case of the analytic Fock space is just the closure of the set of Toeplitz operators.

In \cite{RoVa}, Rozenblum and Vasilevski showed that a Toeplitz operator on a true polyanalytic Fock space $\Fc^2_{(k)}$ is unitarily equivalent to a Toeplitz operator on $\Fc^2$, but with possibly very irregular symbol. They then offer the options of either considering Toeplitz operators on `bad' spaces with `nice' symbols or Toeplitz operators on `nice' spaces with `bad' symbols. Rozenblum and Vasilevski conclude that the second option is more promising. Our (operator algebraic) point of view here is somewhat different as the inner structure of the polyanalytic function spaces does not matter very much in our analysis. Indeed, the band-dominated operators look exactly the same on each true polyanalytic Fock space (in the sense that the algebras are isomorphic) and hence many results can be reduced directly to the analytic case via $\Af$ and $\Af^{\dagger}$. This also lets us circumvent some of the problems in \cite{LuSk}, which prevented a generalization of the Bauer--Isralowitz theorem to polyanalytic Fock spaces. However, the downside is that our construction of the generalized Berezin transform is rather ad-hoc and therefore appears to be less natural. Nevertheless, we manage to obtain a generalization of the Bauer--Isralowitz theorem for polyanalytic Fock spaces using a matrix valued Berezin type transform. It is then not very suprising that our limit operator approach also provides generalizations of other typical results in the area such as the characterization of compact Hankel operators in terms of $\VMO$-functions and the corresponding formula for the essential spectrum of Toeplitz operators. We note that the essential spectrum of polyanalytic Toeplitz operators did not get much attention in the literature so far. The only related (but much weaker) result known to the author is for the polyanalytic Bergman space over the unit disk and due to Wolf \cite{Wolf}. Maybe the most unexpected result of this paper is that the compactness of Hankel operators $H_{f,(k)}$ on $\Fc^2_{(k)}$ does not depend on the order $k$. It is somewhat expected from the work of Rozenblum and Vasilevski \cite{RoVa} that if $H_{f,(1)}$ is compact, then all other Hankel operators $H_{f,(k)}$ are compact as well, but the other direction appears to be rather surprising in this context.

\section{\texorpdfstring{Properties of $\Af$ and $\Af^{\dagger}$}{Properties of A}} \label{sec:Properties_of_A}

We first introduce some more notation that will be needed later on. The algebra of bounded linear operators between two Hilbert spaces $H_1$ and $H_2$ will be denoted by $\Lc(H_1,H_2)$. If $H_1 = H_2$, we will just write $\Lc(H_1)$. Similarly, we will use $\Kc(H_1,H_2)$ and $\Kc(H_1)$ for the compact operators between the respective Hilbert spaces. The identity operator on any Hilbert space will be denoted by $I$. Furthermore, the characteristic function of a set $K \subseteq \C$ will be denoted by $\1_K$. An open ball in $\C$ with midpoint $z$ and radius $r$ will be denoted by $B(z,r)$.

The first property we are going to state is rather obvious, but still worth noting for later.

\begin{prop} \label{prop:shifts_and_projections}
For every $k \in \N$ we have $\Af^{\dagger}P_{(k)}\Af = P_{(k+1)}$.
\end{prop}

\begin{proof}
As $\Af\Af^{\dagger} = I$, $\Af^{\dagger}P_{(k)}\Af$ is an orthogonal projection onto $\Fc^2_{(k+1)}$, hence equal to $P_{(k+1)}$.
\end{proof}

Next, we need the concept of band-dominated operators. The notion originates in the theory of sequence spaces (see e.g.~\cite{RaRoSi}), but has also been introduced to Bergman and Fock spaces a few years ago in order to study compactness and Fredholm problems \cite{FuHa,Hagger_Ball}. A slightly more systematic introduction of band-dominated operators to non-discrete spaces is given in \cite{HaSe}.

\begin{defn} \label{defn:BDO}
An operator $T \in \Lc(L^2(\C,\mu))$ is called a band operator if
\[\sup\set{\dist(K,K') : K,K' \subseteq \C, M_{\1_{K'}}TM_{\1_K} \neq 0} < \infty,\]
where $\dist(K,K') := \inf\limits_{w \in K,z \in K'} \abs{w-z}$ is the distance between the sets $K$ and $K'$. $T$ is called band-dominated if it is the norm limit of a sequence of band operators. The set of band-dominated operators is denoted by $\BDO^2$. An operator $T$ defined on $\Fc^2_{(k)}$ or $\Fc^2_n$ is called band-dominated if $TP_{(k)} \in \BDO^2$ or $TP_n \in \BDO^2$, respectively. The sets of band-dominated operators in $\Lc(\Fc^2_{(k)})$ and $\Lc(\Fc^2_n)$ will be denoted by $\Ac^2_{(k)}$ and $\Ac^2_n$, respectively.
\end{defn}

It turns out that the sets $\BDO^2$, $\Ac^2_{(k)}$ and $\Ac^2_n$ are actually $C^*$-algebras that contain all compact operators; see \cite[Theorems 3.7 and 3.10]{HaSe}. Luckily, for most integral operators it is relatively staightforward to prove their membership in $\BDO^2$. The following lemma is a special case that essentially follows from Young's inequality.

\begin{lem} \label{lem:integral_ops_BDO}
Let $\nu$ be the Gaussian measure defined by $\mathrm{d}\nu(z) := \frac{1}{2\pi}e^{-\frac{1}{2}|z|^2} \mathrm{d}z$, $g \in L^1(\C,\nu)$ and
\[(Tf)(z) := \int_{\C} f(w)g(z-w)e^{z\bar w} \, \mathrm{d}\mu(w)\]
for $z \in \C$ and $f \in L^2(\C,\mu)$. Then $T$ defines a bounded linear operator on $L^2(\C,\mu)$, $T \in \BDO^2$ and $\norm{T} \leq 2\norm{g}_{L^1(\C,\nu)}$.
\end{lem}

\begin{proof}
For $m \geq 0$ we define
\[(T_mf)(z) := \int_{\C} f(w)g(z-w)\1_{B(0,m)}(z-w)e^{z\bar{w}} \, \mathrm{d}\mu(w)\]
for $z \in \C$ and $f \in L^2(\C,\mu)$. This implies
\[\abs{(T-T_m)f(z)}e^{-\frac{1}{2}\abs{z}^2} \leq \frac{1}{\pi}\int_{\C} \abs{f(w)}e^{-\frac{1}{2}\abs{w}^2}\1_{\C \setminus B(0,m)}(z-w)\abs{g(z-w)}e^{-\frac{1}{2}\abs{z-w}^2} \, \mathrm{d}w.\]
Using Young's inequality, we get
\[\norm{(T - T_m)f} \leq 2\norm{f}\norm{g\1_{\C \setminus B(0,m)}}_{L^1(\C,\nu)}.\]
For $m = 0$ we obtain the boundedness of $T$ and the norm estimate. As the operators $T_m$ are obviously band operators, we also get $T \in \BDO^2$.
\end{proof}

With this lemma we can now prove the following important proposition. We actually do not know if $\Af$ and $\Af^{\dagger}$ are in $\BDO^2$ themselves, but for our purposes it is sufficent to know $\Af P_{(k)},\Af^{\dagger}P_{(k)} \in \BDO^2$.

\begin{prop} \label{prop:AP_BDO}
$\Af P_{(k)}$ and $\Af^{\dagger} P_{(k)}$ are contained in $\BDO^2$ for all $k \in \N$.
\end{prop}

\begin{proof}
As $\Af^{\dagger} P_{(k)}$ is the adjoint of $\Af P_{(k+1)}$ and $\BDO^2$ is a $C^*$-algebra, it suffices to show that $\Af P_{(k)} \in \BDO^2$ for all $k \in \N$. $\Af P_{(1)}$ vanishes, so assume $k \geq 2$. For $k \geq 2$ the operator $\Af P_{(k)}$ can be written as an integral operator:
\begin{align*}
\Af P_{(k)}f(z) &= \frac{1}{\sqrt{k-1}} \af P_{(k)}f(z)\\
&= \frac{1}{\sqrt{k-1}} \frac{\partial}{\partial \bar{z}} \int_{\C} f(w)L_{k-1}^0(\abs{z-w}^2)e^{z\bar{w}} \, \mathrm{d}\mu(w)\\
&= -\frac{1}{\sqrt{k-1}} \int_{\C} f(w)(z-w)L_{k-2}^1(\abs{z-w}^2)e^{z\bar{w}} \, \mathrm{d}\mu(w),
\end{align*}
where $f \in L^2(\C,\mu)$, $z \in \C$ and we used $(L_{k-1}^0)' = -L_{k-2}^1$. Choosing $g(z) := -\frac{1}{\sqrt{k-1}}zL_{k-2}^1(\abs{z}^2)$, we obtain $\Af P_{(k)} \in \BDO^2$ by Lemma \ref{lem:integral_ops_BDO}.
\end{proof}

As $\BDO^2$ is a $C^*$-algebra, we also have $P_{(k)} = \Af P_{(k+1)}\Af^{\dagger} P_{(k)} \in \BDO^2$. Obviously, one could also check this directly by the same argument as in Proposition \ref{prop:AP_BDO}.

\begin{cor} \label{cor:P_BDO}
$P_{(k)},P_n \in \BDO^2$ for all $k,n \in \N$.
\end{cor}

As the multiplication operators $M_f$ for $f \in L^{\infty}(\C,\mu)$ are obviously contained in $\BDO^2$, all the Toeplitz and Hankel operators defined in the introduction are band-dominated.

\begin{cor} \label{cor:band_dominated}
Let $k,n \in \N$ and $f \in L^{\infty}(\C,\mu)$. Then the operators $T_{f,(k)}$, $T_{f,n}$, $H_{f,(k)}$ and $H_{f,n}$ are band-dominated.
\end{cor}

We also get the following corollary of Proposition \ref{prop:AP_BDO}. It shows that the algebras of band-dominated operators on each true polyanalytic Fock space are isomorphic. This will allow us to jump back and forth between the spaces and, in particular, obtain a generalization of the Bauer--Isralowitz theorem.

\begin{cor} \label{cor:A^2_isomorphic}
The $C^*$-algebras $\Ac^2_{(k)}$ are isomorphic for $k \in \N$. The isomorphism $\Ac^2_{(1)} \to \Ac^2_{(k)}$ is given by $T \mapsto (\Af^{\dagger})^{k-1}T\Af^{k-1}$.
\end{cor}

For $z \in \C$ the Weyl operators $W_z \from L^2(\C,\mu) \to L^2(\C,\mu)$ are defined by
\[(W_zf)(w) = f(w-z)k_z(w),\]
where the $k_z$ denote the normalized reproducing kernels on $\Fc^2$, that is,
\[k_z(w) := e^{w\bar{z} - \frac{1}{2}\abs{z}^2}.\]
The following properties of Weyl operators are well-known and easy to check: $W_z$ is unitary with $W_z^* = W_{-z}$ and
\begin{equation} \label{eq:W_zW_w}
W_zW_w = e^{-i\Im(z\bar{w})}W_{z+w}.
\end{equation}
Moreover, $W_z$ obviously leaves $\Fc^2_{(1)}$ invariant. Our next proposition shows that all the true polyanalytic Fock spaces are actually left invariant.

\begin{prop} \label{prop:A_commutes_with_Weyl_operators}
We have $\Af W_z = W_z\Af$ and $\Af^{\dagger}W_z = W_z\Af^{\dagger}$ for all $z \in \C$. In particular, $\Fc^2_{(k)}$ is invariant under $W_z$ and we have $P_{(k)}W_z = W_zP_{(k)}$ for all $k \in \N$.
\end{prop}

\begin{proof}
Let $g \in \Fc^2_{(k)}$. Then
\[(\af W_zg)(w) = \frac{\partial}{\partial \bar{w}} \big(g(w-z)k_z(w)\big) = \frac{\partial g}{\partial \bar{w}}(w-z)k_z(w) = (W_z\af g)(w),\]
hence $\Af W_z = W_z\Af$. Taking adjoints yields the other equality since $W_z^* = W_{-z}$. The second claim follows via Proposition \ref{prop:shifts_and_projections} as $\Fc^2_{(1)}$ is invariant under $W_z$.
\end{proof}

Another useful property of the Weyl operators is
\begin{equation} \label{eq:multiplication_shift}
W_{-z}M_fW_z = M_{f(\cdot + z)}
\end{equation}
for all $f \in L^{\infty}(\C,\mu)$ and $z \in \C$; see e.g.~\cite[Lemma 17]{FuHa}.

\section{Limit operator methods} \label{sec:limit_operators}

In this section we will briefly recall the limit operator methods developed in \cite{HaSe} and show how they can be applied to operators on polyanalytic Fock spaces. The main idea of \cite{HaSe} was to formulate a collection of very general assumptions that are needed to characterize compact and Fredholm operators in terms of limit operators. The assumptions are as follows.

\begin{ass}[Space] \label{ass:1}
Let $(X,d)$ be a proper metric space of bounded geometry that satisfies property A'. Assume that $d$ is unbounded and let $\mu$ be a Radon measure on $X$.
\end{ass}

In our case we have $X = \C$, $d$ is the usual Euclidean metric and $\mu$ is the Gaussian measure defined in \eqref{eq:Gaussian}. We will skip the definitions of the more technical terms and just mention that they are more or less trivially satisfied here. For details, we refer to \cite[Example 6.5]{HaSe}.

\begin{ass}[Subspaces and projection] \label{ass:2}
Let $p \in (1,\infty)$ and let $M^p \subseteq L^p(X,\mu)$ be a closed subspace with bounded projection $P \in \BDO^p$. Moreover, assume that $M_{\1_K}P$ and $PM_{\1_K}$ are compact for all compact subsets $K \subset X$.
\end{ass}

For simplicity we only consider the case $p = 2$ here. $M^2$ will be any of the polyanalytic Fock spaces $\Fc^2_{(k)}$ or $\Fc^2_n$. In each case we choose the orthogonal projection discussed in the introduction for $P$. For these we have $P \in \BDO^2$ as shown in Corollary \ref{cor:P_BDO}. That $M_{\1_K}P$ and $PM_{\1_K}$ are compact follows from a Hilbert--Schmidt type argument and will be proven in Theorem \ref{thm:main} below.

\begin{ass}[Shifts] \label{ass:3}
Fix $x_0 \in X$. For $x \in X$ let $\phi_x \from X \to X$ be a bijective isometry with $\phi_x(x_0) = x$. Assume that $\mu \circ \phi_x \ll \mu \ll \mu \circ \phi_x$ and let $h_x$ be a measurable function such that $\abs{h_x}^p = \frac{\mathrm{d}(\mu \circ \phi_x)}{\mathrm{d}\mu}$ $\mu$-almost everywhere. Assume that the maps $x \mapsto \phi_x(y)$ and $x \mapsto h_x(y)$ are continuous for $\mu$-almost every $y \in X$. For $p \in (1,\infty)$ and $x \in X$ let $U_x^p \from L^p(X,\mu) \to L^p(X,\mu)$ be defined by $U_x^pf := (f \circ \phi_x) \cdot h_x$ and assume that $x \mapsto M_{\1_K}U_x^pP(U_x^p)^{-1}M_{\1_{K'}}$ extends continuously to the Stone-\v{C}ech compactification $\beta X$ of $X$ for all compact sets $K,K' \subset X$.
\end{ass}

We choose $x_0 = 0$ and $\phi_x(z) := z+x$ for each $x \in \C$. For $\mu \circ \phi_x$ we get
\[\mathrm{d}(\mu \circ \phi_x)(z) = \frac{1}{\pi}e^{-\abs{z+x}^2} \, \mathrm{d}z = e^{-2\Re(z\bar x) - \abs{x}^2} \mathrm{d}\mu(z)\]
and hence $\mu \circ \phi_x \ll \mu \ll \mu \circ \phi_x$ for all $x \in \C$. For $h_x$ we choose $h_x(z) := e^{-z\bar x - \frac{1}{2}\abs{x}^2} = k_{-x}(z)$. For $U_x^2$ we then have
\[(U_x^2f)(z) = f(z+x)k_{-x}(z) = (W_{-x}f)(z).\]
The last assumption trivially holds as $U_x^2 = W_{-x}$ and $P$ commute in our case by Proposition \ref{prop:A_commutes_with_Weyl_operators}.

We have to admit that in the following the notation is a bit unfortunate as the sign convention is conflicting between \cite{HaSe} and earlier works such as \cite{BaIs} and \cite{FuHa}. We will use the former convention because it makes it compatible with the boundary values of the Berezin transform introduced later and consequently lets some results appear much cleaner (cf.~Lemma \ref{lem:constant_limit_operators} or Lemma \ref{lem:VO_limit_operators} below). Anyway, let $x \in \beta\C \setminus \C$ and choose a net $(z_{\gamma})$ in $\C$ that converges to $x \in \beta\C \setminus \C$. For $H := \Fc^2_{(k)}$ or $H := \Fc^2_n$ and $T \in \Lc(H)$ band-dominated we define
\[T_xg := \lim\limits_{z_{\gamma} \to x} W_{-z_{\gamma}}TW_{z_{\gamma}}g\]
for $g \in H$. This strong limit is guaranteed to exist and independent of the chosen net as shown in \cite[Theorem 4.11]{HaSe}. $T_x$ is called a limit operator of $T \in \Lc(H)$. The main results of \cite{HaSe} are now as follows:

\begin{thm}[Corollary 4.24 of \cite{HaSe}]
Assume that the Assumptions \ref{ass:1}-\ref{ass:3} are satisfied. Then $K \in \Lc(M^p)$ is compact if and only if $K$ is band-dominated and $K_x = 0$ for all $x \in \beta X \setminus X$.
\end{thm}

\begin{thm}[Theorem 4.38 of \cite{HaSe}]
Assume that the Assumptions \ref{ass:1}-\ref{ass:3} are satisfied. Further assume that $T \in \Lc(M^p)$ is band-dominated. Then $T$ is Fredholm if and only if $T_x$ is invertible for all $x \in \beta X \setminus X$.
\end{thm}

Translated to the case at hand we obtain the following theorem as a special case.

\begin{thm} \label{thm:main}
Let $H := \Fc^2_{(k)}$ or $H := \Fc^2_n$ for some $k,n \geq 1$.
\begin{itemize}
	\item[(a)] $K \in \Lc(H)$ is compact if and only if $K$ is band-dominated and $K_x = 0$ for all $x \in \beta\C \setminus \C$.
	\item[(b)] Let $T \in \Lc(H)$ be band-dominated. Then $T$ is Fredholm if and only if $T_x$ is invertible for all $x \in \beta\C \setminus \C$.
\end{itemize}
\end{thm}

\begin{proof}
We only need to show that the Assumptions \ref{ass:1}-\ref{ass:3} are satisfied for the true polyanalytic Fock spaces $\Fc^2_{(k)}$. For $\Fc_n^2 = \bigoplus\limits_{k =1}^n \Fc_{(k)}^2$ the theorem then follows as well because the theory ist stable under orthogonal sums (compare with the remark at the end of \cite[Example 6.8]{HaSe}). Alternatively, the proof below also works for $\Fc_n^2$.

So let $H := \Fc^2_{(k)}$ for some $k \in \N$. After the discussion above, the only condition left to prove is the compactness of $M_{\1_K}P_{(k)}$ and $P_{(k)}M_{\1_K}$ for compact sets $K \subset \C$. We have
\[P_{(k)}M_{\1_{K}}f(z) = \int_K f(w)L_{k-1}^0(\abs{z-w}^2)e^{z\bar{w}} \, \mathrm{d}\mu(w)\]
for $f \in L^2(\C,\mu)$, $z \in \C$. This integral operator is Hilbert--Schmidt because
\begin{align*}
\int_{\C} \int_{K} \abs{L_{k-1}^0(\abs{z-w}^2)e^{z\bar{w}}}^2 \, \mathrm{d}\mu(w) \, \mathrm{d}\mu(z) &= \frac{1}{\pi^2} \int_{\C} \int_{K} \abs{L_{k-1}^0(\abs{z-w}^2)}^2e^{-\abs{z-w}^2} \, \mathrm{d}w \, \mathrm{d}z\\
&= \frac{1}{\pi^2} \int_{\C} \int_{K} \abs{L_{k-1}^0(\abs{z}^2)}^2e^{-\abs{z}^2} \, \mathrm{d}w \, \mathrm{d}z\\
&< \infty.
\end{align*}
A similar argument shows that $M_{\1_{K}}P_{(k)}$ is also Hilbert--Schmidt.
\end{proof}

\section{Generalized Berezin transforms} \label{sec:Berezin}

On $\Fc^2_1 = \Fc^2$ one defines the Berezin transform $\Bc(T)$ of an operator $T \in \Lc(\Fc^2_1)$ as
\[[\Bc(T)](z) := \sp{Tk_z}{k_z}.\]
From the Bauer--Isralowitz theorem we know that the compact operators on $\Fc^2_1$ can be characterized via the Berezin transform. Namely, $T \in \Lc(\Fc^2_1)$ is compact if and only if $T$ is in the Toeplitz algebra and $\Bc(T) \in C_0(\C)$ \cite[Theorem 1.1]{BaIs}. Note that in this case the Toeplitz algebra coincides with the algebra of band-dominated operators $\Ac^2_1$ \cite[Theorem 4.20]{BaFu}.

As $\Fc^2_n$ is also a reproducing kernel Hilbert space, one can of course ask the same question: Can the compact operators be characterized via the Berezin transform? Unfortunately, the answer ist no for $n \geq 2$ (at least within $\Ac^2_n$). Define the normalized reproducing kernels $k_{z,n}$ as usual:
\[k_{z,n}(w) := \frac{K(w,z)}{\norm{K(\cdot,z)}} = \frac{1}{\sqrt{n}}L_{n-1}^1(\abs{z-w}^2)e^{w\bar{z}-\frac{1}{2}|z|^2}.\]
Then
\[\sp{(P_{(1)}-P_{(2)})k_{z,n}}{k_{z,n}} = \sp{P_{(1)}k_{z,n}}{k_{z,n}} - \sp{P_{(2)}k_{z,n}}{k_{z,n}} = \frac{1}{n} - \frac{1}{n} = 0\]
for all $z \in \C$. However, $P_{(1)}-P_{(2)}$ is obviously not compact.

Nevertheless, one can still characterize compactness on $\Fc^2_n$ with a real analytic function, but it turns out to be matrix-valued. The set of complex $n \times n$ matrices will be denoted by $\C^{n \times n}$.

\begin{defn} \label{defn:generalized_Berezin}
Let $k,n \in \N$. For $T \in \Lc(\Fc^2_{(k)})$ we define
\[\Bc_{(k)}(T) \from \C \to \C, \quad [\Bc_{(k)}(T)](z) := \sp{Tl_{z,k}}{l_{z,k}}\]
and for $T \in \Lc(\Fc^2_n)$ we define
\[\Bc_n(T) \from \C \to \C^{n \times n}, \quad [\Bc_n(T)](z) := \begin{pmatrix} \sp{Tl_{z,1}}{l_{z,1}} & \hdots & \sp{Tl_{z,n}}{l_{z,1}} \\ \vdots & & \vdots \\ \sp{Tl_{z,1}}{l_{z,n}} & \hdots & \sp{Tl_{z,n}}{l_{z,n}} \end{pmatrix},\]
where $l_{z,k} \in \Fc^2_{(k)}$ is given by $l_{z,k}(w) = \frac{1}{\sqrt{(k-1)!}}(\bar{w}-\bar{z})^{k-1}e^{w\bar{z}-\frac{1}{2}|z|^2}$. We will also use the notation $\widetilde{T} $ for the Berezin transform of $P_1T|_{\Fc^2_1} \from \Fc^2_1 \to \Fc^2_1$. For $f \in L^{\infty}(\C,\mu)$ we will use the abbreviations $\Bc_n(f) := \Bc_n(T_{f,n})$, $\Bc_{(k)}(f) := \Bc_{(k)}(T_{f,(k)})$ and $\tilde{f} := \Bc_1(f)$.
\end{defn}

Note that $l_{z,k}$ is indeed in $\Fc^2_{(k)}$ as
\begin{equation} \label{eq:k_and_l}
l_{z,k}(w) = (\Af^{\dagger})^{k-1}k_z(w),
\end{equation}
which is easily seen by induction. These generalized Berezin transforms thus inherit all the basic properties of the usual Berezin transform $\Bc = \Bc_1$. In particular, $\Bc_{(k)} \from \Lc(\Fc^2_{(k)}) \to C(\C)$ and $\Bc_n \from \Lc(\Fc^2_n) \to C(\C \to \C^{n \times n})$ are injective bounded linear operators and their images consist of real analytic, Lipschitz continuous functions.

This also makes it evident why this is suitable to characterize compactness. Namely, $T \in \Lc(\Fc^2_n)$ is compact if and only if $P_{(k)}T|_{\Fc^2_{(j)}} \in \Lc(\Fc^2_{(j)},\Fc^2_{(k)})$ is compact for all $j,k = 1,\ldots,n$, which is equivalent to $P_{(1)}\Af^{k-1}T(\Af^{\dagger})^{j-1}|_{\Fc^2_{(1)}} \in \Lc(\Fc^2_{(1)})$ being compact. As the Berezin transform of $P_{(1)}\Af^{k-1}T(\Af^{\dagger})^{j-1}|_{\Fc^2_{(1)}}$ is given by
\[\sp{P_{(1)}\Af^{k-1}T(\Af^{\dagger})^{j-1}k_z}{k_z} = \sp{Tl_{z,j}}{l_{z,k}},\]
the following theorem now follows directly from the Bauer--Isralowitz theorem on $\Fc^2$ \cite[Theorem 1.1]{BaIs}, \cite[Theorem 4.20]{BaFu} mentioned above and Corollary \ref{cor:A^2_isomorphic}.

\begin{thm} \label{thm:compact_Berezin}
Let $j,k,n \in \N$.
\begin{itemize}
	\item[(a)] $T \in \Lc(\Fc^2_{(j)},\Fc^2_{(k)})$ is compact if and only if $TP_{(j)} \in \BDO^2$ and $z \mapsto \sp{Tl_{z,j}}{l_{z,k}}$ is in $C_0(\C)$.
	\item[(b)] $T \in \Lc(\Fc^2_n)$ is compact if and only if $T \in \Ac^2_n$ and $\Bc_n(T) \in C_0(\C \to \C^{n \times n})$.
\end{itemize}
\end{thm}

As $\Af^{\dagger}$ commutes with the Weyl operators (see Proposition \ref{prop:A_commutes_with_Weyl_operators}), $\Bc_n$ also preserves the shift action.

\begin{lem} \label{lem:Weyl_operators_and_Berezin_transform}
Let $T \in \Lc(\Fc^2_n)$. We have
\[[\Bc_n(W_{-\zeta}TW_{\zeta})](z) = [\Bc_n(T)](z+\zeta)\]
for all $\zeta,z \in \C$.
\end{lem}

\begin{proof}
Let $j,k \in \set{1,\ldots,n}$. By \eqref{eq:W_zW_w} and \eqref{eq:k_and_l} we have
\[W_{\zeta}l_{z,j} = W_{\zeta}(\Af^{\dagger})^{j-1}k_z = e^{-i\Im(\zeta\bar{z})}W_{z+\zeta}(\Af^{\dagger})^{j-1}\1,\]
hence
\begin{align*}
\sp{W_{-\zeta}TW_{\zeta}l_{z,j}}{l_{z,k}} &= \sp{TW_{z+\zeta}(\Af^{\dagger})^{j-1}\1}{W_{z+\zeta}(\Af^{\dagger})^{k-1}\1} = \sp{T(\Af^{\dagger})^{j-1}k_{z+\zeta}}{(\Af^{\dagger})^{k-1}k_{z+\zeta}}\\
&= \sp{Tl_{z+\zeta,j}}{l_{z+\zeta,k}}.\qedhere
\end{align*}
\end{proof}

Let us define the oscillation of a bounded continuous function $f \from \C \to \C^{n \times n}$ at $z$ as
\[\Osc_z(f) := \sup\set{\norm{f(z)-f(w)} : \abs{z-w} \leq 1},\]
where $\norm{\cdot}$ is of course the usual matrix norm on $\C^{n \times n}$ induced by the Euclidean norm on $\C^n$. As $z \mapsto \Osc_z(f)$ is continuous and bounded, we can extend it to the Stone-\v{C}ech compactification of $\C$. We will use the notation $\Osc_x(f)$ for the extension evaluated at some point $x \in \beta\C$.

We say that a bounded continuous function $f \from \C \to \C^{n \times n}$ has vanishing oscillation and write $f \in \VO(\C \to \C^{n \times n})$ if
\[\lim\limits_{\abs{z} \to \infty} \Osc_z(f) = 0.\]
For $n = 1$ we just write $f \in \VO(\C)$ as usual. Note that $f \in \VO(\C \to \C^{n \times n})$ if and only if all of its matrix entries are in $\VO(\C)$.

The following is an adaptation of \cite[Theorem 36]{Hagger_BSD}. Here, $I_n$ denotes he $n \times n$ identity matrix.

\begin{lem} \label{lem:constant_limit_operators}
Let $\lambda \in \C$ and $x \in \beta\C \setminus \C$.
\begin{itemize}
	\item[(a)] For $T \in \Ac^2_{(k)}$ we have $T_x = \lambda I$ if and only if $[\Bc_{(k)}(T)](x) = \lambda$ and $\Osc_x(\Bc_{(k)}(T)) = 0$.
	\item[(b)] For $T \in \Ac^2_n$ we have $T_x = \lambda I$ if and only if $[\Bc_n(T)](x) = \lambda I_n$ and $\Osc_x(\Bc_n(T)) = 0$.
\end{itemize}
\end{lem}

\begin{proof}
The proofs of (a) and (b) are identical, so we only show (b).

Assume that $T_x = \lambda I$. Choose a net $(z_{\gamma})$ in $\C$ that converges to $x$. By Lemma \ref{lem:Weyl_operators_and_Berezin_transform}, it follows
\begin{equation} \label{eq:constant_limit_operators}
\lim\limits_{z_{\gamma} \to x} [\Bc_n(T)](z+z_{\gamma}) = \lim\limits_{z_{\gamma} \to x} [\Bc_n(W_{-z_{\gamma}}TW_{z_{\gamma}})](z) = [\Bc_n(T_x)](z) = \lambda I_n
\end{equation}
for all $z \in \C$. In particular, choosing $z = 0$, we obtain $[\Bc_n(T)](x) = \lambda I_n$.

Now assume that $\Osc_x(\Bc_n(T)) \neq 0$. Then there is an $\epsilon > 0$ and a net $(z_{\gamma})$ converging to $x$ such that $\Osc_{z_{\gamma}}(\Bc_n(T)) > 2\epsilon$ for all $\gamma$. For every $\gamma$ we may choose a $w_{\gamma} \in \C$ with $\abs{z_{\gamma}-w_{\gamma}} \leq 1$ such that
\[\norm{[\Bc_n(T)](z_{\gamma}) - [\Bc_n(T)](w_{\gamma})} > \epsilon.\]
Without loss of generality we may assume that the net $(z_{\gamma}-w_{\gamma})$ converges to some $z \in \overline{B(0,1)}$. The Lipschitz continuity of $\Bc_n(T)$ then implies that there exists a $C \geq 0$ such that
\begin{align*}
\norm{[\Bc_n(T)](z_{\gamma}) - [\Bc_n(T)](w_{\gamma})} \!&\leq\! \norm{[\Bc_n(T)](z_{\gamma}) - [\Bc_n(T)](z_{\gamma}-z)} + \norm{[\Bc_n(T)](z_{\gamma}-z) - [\Bc_n(T)](w_{\gamma})}\\
\!&\leq\! \norm{[\Bc_n(T)](z_{\gamma}) - [\Bc_n(T)](z_{\gamma}-z)} + C\abs{z_{\gamma} - z - w_{\gamma}}.
\end{align*}
Using \eqref{eq:constant_limit_operators}, this tends to $0$, which is a contradiction. Thus $\Osc_x(\Bc_n(T)) = 0$.

Now assume that $\Osc_x(\Bc_n(T)) = 0$ and $[\Bc_n(T)](x) = \lambda I_n$. Choose a net $(z_{\gamma})$ in $\C$ that converges to $x$. Then, by Lemma \ref{lem:Weyl_operators_and_Berezin_transform} again,
\[[\Bc_n(T_x)](0) = \lim\limits_{z_{\gamma} \to x} [\Bc_n(W_{-z_{\gamma}}TW_{z_{\gamma}})](0) = \lim\limits_{z_{\gamma} \to x} [\Bc_n(T)](z_{\gamma}) = [\Bc_n(T)](x) = \lambda I_n.\]
Moreover,
\begin{align*}
\norm{[\Bc_n(T_x)](w) - [\Bc_n(T_x)](0)} &= \lim\limits_{z_{\gamma} \to x} \norm{[\Bc_n(W_{-z_{\gamma}}TW_{z_{\gamma}})](w) - [\Bc_n(W_{-z_{\gamma}}TW_{z_{\gamma}})](0)}\\
&= \lim\limits_{z_{\gamma} \to x} \norm{[\Bc_n(T)](w+z_{\gamma}) - [\Bc_n(T)](z_{\gamma})}\\
&\leq \lim\limits_{z_{\gamma} \to x} \Osc_{z_{\gamma}}(\Bc_n(T))\\
&= 0
\end{align*}
for $\abs{w} \leq 1$. As $\Bc_n(T_x)$ is real analytic, the identity theorem implies $[\Bc_n(T_x)](w) = \lambda I_n$ for all $w \in \C$. The injectivity of $\Bc_n$ thus shows $T_x = \lambda I$ as expected.
\end{proof}

\begin{lem} \label{lem:VO_limit_operators}
Let $x \in \beta\C \setminus \C$ and $f \in \VO(\C)$. Then $(T_{f,n})_x = f(x) I$ for all $n \in \N$.
\end{lem}

\begin{proof}
Choose a net $(z_{\gamma})$ in $\C$ that converges to $x$. By Combining \eqref{eq:multiplication_shift} and Proposition \ref{prop:A_commutes_with_Weyl_operators}, we get
\[W_{-z_{\gamma}}T_{f,n}W_{z_{\gamma}} = T_{f(\cdot + z_{\gamma}),n}.\]
For $\abs{w} \leq 1$ we have
\begin{align*}
\abs{f(w + z_{\gamma}) - f(x)} &\leq \abs{f(w + z_{\gamma}) - f(z_{\gamma})} + \abs{f(z_{\gamma}) - f(x)}\\
&\leq \Osc_{z_{\gamma}}(f) + \abs{f(z_{\gamma}) - f(x)}.
\end{align*}
Since $f \in \VO(\C)$, this converges to $0$ uniformly in $w$. A straightforward induction argument now shows that the net $(f(\cdot + z_{\gamma}))$ converges to the constant function $f(x)$ uniformly on compact sets. This implies $T_{f(\cdot + z_{\gamma}),n} \to f(x) I$ in the strong operator topology.
\end{proof}

We now combine the previous two lemmas to our next theorem, which can be viewed as a generalization of Theorem \ref{thm:main} (a). The formula for the essential spectrum is well-known in case $k = 1$ (see \cite[Theorem 19]{BeCo}).

\begin{thm} \label{thm:VO+K_1}
Let $T \in \Ac^2_{(k)}$. The following are equivalent:
\begin{itemize}
	\item[(i)] Every limit operator of $T$ is a multiple of the identity.
	\item[(ii)] $T = T_{f,(k)} + K$, where $f \in \VO(\C)$ and $K$ is compact.
\end{itemize}
In that case $f$ can be chosen as $\Bc_{(k)}(T)$ and we have
\[\spec_{\ess}(T) = [\Bc_{(k)}(T)](\beta\C \setminus \C).\]
\end{thm}

\begin{proof}
Assume that every limit operator of $T$ is a multiple of the identity. Then Lemma \ref{lem:constant_limit_operators} implies $f := \Bc_{(k)}(T) \in \VO(\C)$. It remains to show that $T - T_{f,(k)}$ is compact. Let $x \in \beta\C \setminus \C$. We know from Lemma \ref{lem:VO_limit_operators} that $(T_{f,n})_x = f(x) I$ for every $n \in \N$. Since $P_{(k)}$ commutes with the Weyl operators (see Proposition \ref{prop:A_commutes_with_Weyl_operators}) and $T_{f,(k)} = P_{(k)}T_{f,n}|_{\Fc^2_{(k)}}$ for any $n \geq k$, this means that we also have $(T_{f,(k)})_x = f(x) I$. On the other hand, we also have $T_x = f(x) I$ by Lemma \ref{lem:constant_limit_operators}. It follows $(T - T_{f,(k)})_x = 0$. As $x \in \beta\C \setminus \C$ was arbitrary, Theorem \ref{thm:main} (a) implies that $T - T_{f,(k)}$ is indeed compact. Moreover, we get the formula $\spec_{\ess}(T) = [\Bc_{(k)}(T)](\beta\C \setminus \C)$ via Theorem \ref{thm:main} (b).

Now assume that $T = T_{f,(k)} + K$ for some $f \in \VO(\C)$ and $K \in \Kc(\Fc^2_{(k)})$. Using Theorem \ref{thm:main} (a) and Lemma \ref{lem:VO_limit_operators} again, we obtain
\[T_x = (T_{f,(k)})_x + K_x = f(x) I\]
for every $x \in \beta\C \setminus \C$.
\end{proof}

Next, we want to prove a version of this result for operators $T \in \Ac^2_n$. However, a priori it is not clear what function $f$ one should take for the decomposition $T = T_{f,n} + K$ because $\Bc_n(T)$ is matrix valued. It turns out that any of the functions $\Bc_{(k)}(T) := \Bc_{(k)}(P_{(k)}T|_{\Fc^2_{(k)}})$, $k = 1,\ldots,n$ works because they only differ by a $C_0$-function in that case.

\begin{lem} \label{lem:Berezin_difference_C_0}
Let $T \in \Ac^2_n$ and assume that every limit operator of $T$ is a multiple of the identity. Then
\[g := \widetilde{T} - \Bc_{(k)}(T) \in C_0(\C)\]
for all $k = 1,\ldots,n$. In particular, $T_{g,(k)}$ is compact.
\end{lem}

\begin{proof}
Let $k \in \set{1,\ldots,n}$, $x \in \beta\C \setminus \C$ and $T_x = \lambda I$ for some $\lambda \in \C$. This implies that $(P_{(k)}T|_{\Fc^2_{(k)}})_x = \lambda I$ as well because $P_{(k)}$ commutes with $W_z$ by Proposition \ref{prop:A_commutes_with_Weyl_operators}. Thus $[\Bc_{(k)}(T)](x) = \lambda$ for all $k = 1,\ldots,n$ by Lemma \ref{lem:constant_limit_operators}. As $x \in \beta\C \setminus \C$ was arbitrary, $\widetilde{T} = \Bc_{(1)}(T)$ and $\Bc_{(k)}(T)$ agree on $\beta\C \setminus \C$, which proves the first part of the lemma. That $C_0$-functions produce compact Toeplitz operators is easy to show directly, but also follows from Theorem \ref{thm:main} (a) and Lemma \ref{lem:VO_limit_operators}.
\end{proof}

We can now formulate a version of Theorem \ref{thm:VO+K_1} for $T \in \Ac^2_n$.

\begin{thm} \label{thm:VO+K_2}
Let $T \in \Ac^2_n$. The following are equivalent:
\begin{itemize}
	\item[(i)] Every limit operator of $T$ is a multiple of the identity.
	\item[(ii)] $T = T_{f,n} + K$, where $f \in \VO(\C)$ and $K$ is compact.
\end{itemize}
In that case $f$ can be chosen as $\widetilde{T}$ and we have
\[\spec_{\ess}(T) = \widetilde{T}(\beta\C \setminus \C).\]
\end{thm}

\begin{proof}
Assume (i) and write $T = \sum\limits_{j,k = 1}^n P_{(k)}TP_{(j)}$. Combining Theorem \ref{thm:VO+K_1} with Lemma \ref{lem:Berezin_difference_C_0} we get that $P_{(k)}TP_{(k)} - P_{(k)}T_{\widetilde{T},n}P_{(k)}$ is compact and $\widetilde{T} \in \VO(\C)$. $P_{(k)}TP_{(j)}$ and $P_{(k)}T_{\widetilde{T},n}P_{(j)}$ are also compact for $j \neq k$. In both cases this follows from Theorem \ref{thm:main} as all limit operators are zero, respectively. We infer that $T - T_{\widetilde{T},n}$ is compact, which implies (ii).

The other direction and the formula for the essential spectrum follow from Theorem \ref{thm:main} and Lemma \ref{lem:VO_limit_operators} just like in the proof of Theorem \ref{thm:VO+K_1}.
\end{proof}

\section{Compact Toeplitz and Hankel operators} \label{sec:Toeplitz_and_Hankel}

In \cite{RoVa} it was observed that a Toeplitz operator on a true polyanalytic Fock space $\Fc^2_{(k)}$ is unitarily equivalent to a Toeplitz operator on the analytic Fock space $\Fc^2$ with a much more irregular, possibly distributional symbol. After this observation, Rozenblum and Vasilevski offer a choice of considering ``operators with nice symbols in `bad' spaces or operators in nice spaces with `bad' symbols''. Conversely, one would therefore expect that if we take a very good symbol, e.g.~a symbol that induces a compact Toeplitz operator on $\Fc^2$, then this symbol would also induce an in this sense very good (or even better) Toeplitz operator on the polyanalytic spaces. The next result thus may not be very surprising and in fact follows directly from our considerations above. A related result was proven by different means in \cite[Proposition 4.6]{LuSk}.

\begin{thm} \label{thm:compact_Toeplitz}
Let $k,n \in \N$ and $f \in L^{\infty}(\C,\mu)$. $T_{f,1}$ is compact if and only if $T_{f,n}$ is compact. In particular, if $T_{f,(1)}$ is compact, then every $T_{f,(k)}$ is compact as well.
\end{thm}

\begin{proof}
As $T_{f,1}$ is a compression of $T_{f,n}$, it is clear that if $T_{f,n}$ is compact, then $T_{f,1}$ is necessarily compact as well.

So assume that $T_{f,1}$ is compact. By Theorem \ref{thm:compact_Berezin}, we need to show that $\Bc_n(T_{f,n}) \in C_0(\C \to \C^{n \times n})$, that is, $\lim\limits_{\abs{z} \to \infty} \sp{T_{f,n}l_{z,j}}{l_{z,k}} = 0$ for all $j,k = 1,\ldots,n$. We have
\begin{align*}
\sp{T_{f,n}l_{z,j}}{l_{z,k}} &= \frac{1}{\sqrt{(j-1)!}\sqrt{(k-1)!}}\int_{\C} f(w)(\bar{w}-\bar{z})^{j-1}(w-z)^{k-1}e^{w\bar{z}+\bar{w}z-|z|^2} \, \mathrm{d}\mu(w)\\
&= \langle T_{f,1}\hat{l}_{z,k} , \hat{l}_{z,j} \rangle,
\end{align*}
where $\hat{l}_{z,k}(w) = \frac{1}{\sqrt{(k-1)!}}(w-z)^{k-1}e^{w\bar{z}-\frac{1}{2}|z|^2}$. Note that $\hat{l}_{z,k} \in \Fc^2_1$ and
\[\hat{l}_{z,k} = W_zm_k,\]
where $m_k$ is the monomial given by $m_k(w) = \frac{1}{\sqrt{(k-1)!}}w^{k-1}$. We thus have
\[\sp{T_{f,n}l_{z,j}}{l_{z,k}} = \sp{T_{f,1}W_zm_k}{W_zm_j} = \sp{W_{-z}T_{f,1}W_zm_k}{m_j},\]
which converges to $0$ as $\abs{z} \to \infty$ by Theorem \ref{thm:main} (a).
\end{proof}

Combining this result with Theorem \ref{thm:VO+K_2}, we obtain the following generalization. It particularly applies to Toeplitz operators of the form $\lambda I + K$ with $K \in \Kc(\Fc^2_n)$, which is the case considered in \cite[Proposition 4.6]{LuSk}.

\begin{cor}
Let $n \in \N$ and $f \in L^{\infty}(\C,\mu)$. Every limit operator of $T_{f,1}$ is a multiple of the identity if and only if every limit operator of $T_{f,n}$ is a multiple of the identity.
\end{cor}

\begin{proof}
The ``if'' direction is again obvious. So assume that every limit operator of $T_{f,1}$ is a multiple of the identity. Theorem \ref{thm:VO+K_2} implies that $\tilde{f} \in \VO$ and $T_{f-\tilde{f},1}$ is compact. This means that $T_{f-\tilde{f},n}$ is also compact by Theorem \ref{thm:compact_Toeplitz}. Using Theorem \ref{thm:VO+K_2} again, we see that every limit operator of $T_{f,n}$ is a multiple of the identity as well.
\end{proof}

Next, we turn our attention to Hankel operators. The theory developed in Section \ref{sec:limit_operators} cannot be applied directly to them because they do not map into a polyanalytic Fock space. This can of course be circumvented by considering $H_{f,n}^*H_{f,n} \in \Lc(\Fc^2_n)$ instead, which we will do later on. However, we will take a slightly more general approach first which provides a compactness characterization for all operators acting on $L^2(\C,\mu)$. For this we use the decomposition
\[L^2(\C,\mu) = \bigoplus\limits_{k = 1}^\infty \Fc^2_{(k)} \cong \ell^2(\N,\Fc^2)\]
again. The (non-commutative) algebra $\ell^{\infty}(\N,\Lc(\Fc^2))$ acts on $\ell^2(\N,\Fc^2)$ via multiplication, that is, if $g \in \ell^{\infty}(\N,\Lc(\Fc^2))$, then
\[m_g \from \ell^2(\N,\Fc^2) \to \ell^2(\N,\Fc^2), \quad (m_gf)(k) = g(k)f(k)\]
for $f \in \ell^2(\N,\Fc^2)$, $k \in \N$. Moreover, we have the usual shift operators $V$ acting on $\ell^2(\N,\Fc^2)$ defined by
\[(Vf)(1) := 0, \quad (Vf)(k+1) := f(k)\]
for $f \in \ell^2(\N,\Fc^2)$, $k \in \N$. One can now define band-dominated operators on $\ell^2(\N,\Fc^2)$ with respect to their matrix structure as follows.

\begin{defn}[Definition 2.1.5 in \cite{RaRoSi}]
Let $\omega \in \N$. Operators of the form
\[T = \sum\limits_{j = 0}^{\omega} m_{g_j}V^j + \sum\limits_{j = 1}^{\omega} m_{g_{-j}}(V^*)^j\]
with $g_j \in \ell^{\infty}(\N,\Lc(\Fc^2))$ are called band operators on $\ell^2(\N,\Fc^2)$. Operators obtained as the norm limit of a sequence of band operators are then called band-dominated on $\ell^2(\N,\Fc^2)$.
\end{defn}

One could of course also define band-dominated operators on $\ell^2(\N,\Fc^2)$ just like we did for $L^2(\C,\mu)$ in Definition \ref{defn:BDO} and it is not difficult to show that this would be an equivalent definition (see e.g.~\cite[Theorem 2.1.6]{RaRoSi}). However, using this version makes it a little more obvious why these operators are called band-dominated and it will also be more straightforward to use in the proof of the next theorem. We emphasize that it is important not to confuse the two notions of band-dominated operators, though, which is why we added the suffix ``on $\ell^2(\N,\Fc^2)$''. Indeed, they act on different spaces ($L^2(\C,\mu)$ vs.~$\ell^2(\N,\Fc^2)$) and are not equivalent via the isomorphism
\[U \from L^2(\C,\mu) \to \ell^2(\N,\Fc^2).\]

\begin{ex}~ \label{ex:counterexample}
\begin{itemize}
	\item[(a)] 	Consider the Hankel operator $H_{f,1}$ for $f(z) = e^{i\abs{z}^2}$. Then $H_{f,1}$ is band-dominated by Corollary \ref{cor:band_dominated}. We claim that $UH_{f,1}U^{-1}$ is not band-dominated on $\ell^2(\N,\Fc^2)$. So assume by contradiction that $UH_{f,1}U^{-1}$ is band-dominated on $\ell^2(\N,\Fc^2)$. Then necessarily $\norm{(I-P_n)H_{f,1}P_1} \to 0$ as $n \to \infty$ by \cite[Propositions 1.20 and 1.48]{Lindner}. The Toeplitz operator $T_{f,1}$ is compact, which can be deduced directly from its eigenvalues (see \cite[Example (A)]{BaCoHa}). By Theorem \ref{thm:compact_Toeplitz}, this means that for every $n \in \N$ the operator $P_nM_fP_1 = T_{f,n}P_1$ is compact as well. Combining this with $\norm{(I-P_n)H_{f,1}P_1} \to 0$ shows that $H_{f,1}P_1$ is also compact. But this would mean that
	\[P_1 = P_1M_{\bar f}M_fP_1 = P_1M_{\bar f}H_{f,1}P_1 + P_1M_{\bar f}T_{f,1}P_1,\]
	is compact, which is obviously a contradiction. This shows that $UH_{f,1}U^{-1}$ is not band-dominated on $\ell^2(\N,\Fc^2)$.
	\item[(b)] To give an operator $T \notin \BDO^2$ such that $UTU^{-1}$ is band-dominated on $\ell^2(\N,\Fc^2)$ is much simpler; many such examples can be found in the literature. As a specific example we mention the operator
	\[T \from L^2(\C,\mu) \to L^2(\C,\mu), \quad (Tf)(z) = f(-z).\]
	$T$ leaves every $\Fc^2_{(k)}$ invariant, which implies that $UTU^{-1}$ is a band operator (with just one diagonal) on $\ell^2(\N,\Fc^2)$. On the other hand,
	\[\widetilde{T}(z) = \sp{Tk_z}{k_z} = \sp{k_{-z}}{k_z} = e^{-2\abs{z}^2} \to 0\]
	as $\abs{z} \to \infty$. So if $T$ was band-dominated, then $T$ restricted to $\Fc^2_1$ would be compact by Theorem \ref{thm:compact_Berezin}, which it obviously is not. Further, maybe more interesting examples are provided in \cite[Example 2]{BaFu}, for instance.
\end{itemize}
\end{ex}

The next theorem now provides a limit operator type characterization of compact operators on $L^2(\C,\mu)$.

\begin{thm} \label{thm:compact_on_L^2}
Let $T \in \Lc(L^2(\C,\mu))$. Then $T$ is compact if and only if $T \in \BDO^2$, $UTU^{-1}$ is band-dominated on $\ell^2(\N,\Fc^2)$, $W_{-z}TW_z \to 0$ strongly as $\abs{z} \to \infty$ and
\begin{equation} \label{eq:compact_on_L^2}
\lim\limits_{k \to \infty} \norm{P_n\Af^kT(\Af^{\dagger})^kP_n} = 0
\end{equation}
for every $n \in \N$.
\end{thm}

\begin{proof}
Define
\[\Kc(\ell^2(\N,\Fc^2),\Pc) := \set{K \in \Lc(\ell^2(\N,\Fc^2)) : \|K-\hat{P}_nK\| + \|K-K\hat{P}_n\| \to 0 \text{ as } n \to \infty},\]
where $\hat{P}_n := UP_nU^{-1}$ is the orthogonal projection onto $\ell^2(\set{1,\ldots,n},\Fc^2)$. We will first show that $K \in \Kc(\ell^2(\N,\Fc^2),\Pc)$ if and only if $K$ is band-dominated on $\ell^2(\N,\Fc^2)$ and
\begin{equation} \label{eq:compact_on_L^2_2}
\lim\limits_{k \to \infty} \norm{\hat{P}_n(V^*)^kKV^k\hat{P}_n} = 0.
\end{equation}
This follows from a standard limit operator argument; nevertheless, we provide some details for the convenience of the reader. First observe that if $K$ belongs to $\Kc(\ell^2(\N,\Fc^2),\Pc)$, then
\begin{equation} \label{eq:compact_on_L^2_3}
\|K - \hat{P}_nK\hat{P}_n\| \leq \|K - \hat{P}_nK\| + \|\hat{P_n}K - \hat{P_n}K\hat{P}_n\| \to 0
\end{equation}
as $n \to \infty$. $\hat{P}_nK\hat{P}_n$ is obviously a band operator on $\ell^2(\N,\Fc^2)$ and so all $K \in \Kc(\ell^2(\N,\Fc^2),\Pc)$ are band-dominated on $\ell^2(\N,\Fc^2)$. Moreover, since $\hat{P}_nV^k = 0$ for $k \geq n$, we have
\[\lim\limits_{k \to \infty} \norm{(V^*)^kKV^k} = \lim\limits_{k \to \infty} \norm{(V^*)^k(K-K\hat{P}_n)V^k} = 0.\]
In particular, $\lim\limits_{k \to \infty} \norm{\hat{P}_n(V^*)^kKV^k\hat{P}_n} = 0$.

Conversely, assume that $K$ is a band operator on $\ell^2(\N,\Fc^2)$ and satisfies $\norm{\hat{P}_n(V^*)^kKV^k\hat{P}_n} \to 0$ as $k \to \infty$. Write $K$ as $\sum\limits_{j = 0}^{\omega} m_{g_j}V^j + \sum\limits_{j = 1}^{\omega} m_{g_{-j}}(V^*)^j$ for some $g_j \in \ell^{\infty}(\N,\Lc(\Fc^2))$. The condition $\norm{\hat{P}_n(V^*)^kKV^k\hat{P}_n} \to 0$ implies that $\norm{m_{g_j}(k)} \to 0$ for $k \to \infty$, $j \in \set{-\omega,\ldots,\omega}$. It follows
\[\norm{m_{g_j} - \hat{P}_nm_{g_j}} = \norm{m_{g_j} - m_{g_j}\hat{P}_n} \leq \sup\limits_{k \geq n+1} \norm{m_{g_j}(k)} \to 0\]
as $k \to \infty$ and thus $m_{g_j} \in \Kc(\ell^2(\N,\Fc^2),\Pc)$ for each $j \in \set{-\omega,\ldots,\omega}$. As $\Kc(\ell^2(\N,\Fc^2),\Pc)$ is an algebra and invariant under multiplying by $V$ or $V^*$, this shows $K \in \Kc(\ell^2(\N,\Fc^2),\Pc)$. $\Kc(\ell^2(\N,\Fc^2),\Pc)$ is also closed and therefore the same conclusion holds if $K$ is assumed to be a band-dominated instead of a band operator.

We already observed $V = U\Af^{\dagger}U^{-1}$ in the introduction and so \eqref{eq:compact_on_L^2_2} for $K = UTU^{-1}$ is the same as \eqref{eq:compact_on_L^2}. Hence, it suffices to prove that $T$ is compact if and only if $T \in \BDO^2$, $UTU^{-1} \in \Kc(\ell^2(\N,\Fc^2))$ and $W_{-z}TW_z \to 0$ strongly as $\abs{z} \to \infty$. So assume that $T$ is compact. Then $T \in \BDO^2$ is \cite[Theorem 3.7 (d)]{HaSe} and $W_{-z}TW_z \to 0$ as $\abs{z} \to \infty$ follows from \cite[Proposition 4.20]{HaSe}. Moreover, it is clear that $\Kc(\ell^2(\N,\Fc^2))$ contains all compact operators as $\hat{P}_n \to I$ in the strong operator topology. Conversely, assume that $T \in \BDO^2$, $UTU^{-1} \in \Kc(\ell^2(\N,\Fc^2))$ and $W_{-z}TW_z \to 0$ strongly as $\abs{z} \to \infty$. By \eqref{eq:compact_on_L^2_3}, $\norm{T - P_nTP_n} \to 0$ as $n \to \infty$. As $P_n$ and $W_z$ commute (see Proposition \ref{prop:A_commutes_with_Weyl_operators}), we have $W_{-z}P_nTP_nW_z \to 0$ strongly as $\abs{z} \to \infty$. By Theorem \ref{thm:main} (a), $P_nTP_n$ is compact for every $n \in \N$. This shows that $T$ is compact as well, which completes the proof.
\end{proof}

For operators $T$ that satisfy either $\norm{T(I-P_n)} \to 0$ or $\norm{(I-P_n)T} \to 0$ as $n \to \infty$, such as Hankel operators, Theorem \ref{thm:compact_on_L^2} simplifies significantly.

\begin{cor}
Assume that $T \in \Lc(L^2(\C,\mu))$ satisfies either $\norm{T(I-P_n)} \to 0$ or $\norm{(I-P_n)T} \to 0$ as $n \to \infty$. Then $T$ is compact if and only if $T \in \BDO^2$ and $W_{-z}TW_z \to 0$ strongly as $\abs{z} \to \infty$.
\end{cor}

\begin{proof}
As $T$ is compact if and only if $T^*T$ is compact, is suffices to consider $T^*T$. We only prove the case when $\norm{T(I-P_n)} \to 0$; the other case is similar. The condition $\norm{T(I-P_n)} \to 0$ implies
\[\norm{(I-P_n)T^*T} = \norm{T^*T(I-P_n)} \to 0\]
and therefore $UT^*TU^{-1} \in \Kc(\ell^2(\N,\Fc^2),\Pc)$. The result now follows by the same arguments as in the last paragraph of the proof of Theorem \ref{thm:compact_on_L^2}.
\end{proof}

The next corollary is now immediate. It can be viewed as a partial generalization of \cite[Theorem 3.1]{HaVi}.

\begin{cor} \label{cor:Hankel_compact}
Let $f \in L^{\infty}(\C)$ and $k,n \in \N$.
\begin{itemize}
	\item[(a)] $H_{f,(k)}$ is compact if and only if $W_{-z}H_{f,(k)}W_z \to 0$ strongly as $\abs{z} \to \infty$.
	\item[(b)] $H_{f,n}$ is compact if and only if $W_{-z}H_{f,n}W_z \to 0$ strongly as $\abs{z} \to \infty$.
\end{itemize}
\end{cor}

Of course, we also have the classical characterization of compact Hankel operators in terms of the Berezin transform. Quite suprisingly, the compactness of $H_{f,(k)}$ does not depend on $k$. For $f \in L^{\infty}(\C,\mu)$ we say that $f \in \VMO(\C)$ if
\[\widetilde{\abs{f}^2} - |\tilde{f}|^2 \in C_0(\C).\]

\begin{thm} \label{thm:compact_Hankel}
Let $k,n \in \N$ and $f \in L^{\infty}(\C)$. The following are equivalent:
\begin{itemize}
	\item[(a)] $H_{f,(k)}$ is compact,
	\item[(b)] $H_{f,n}$ is compact,
	\item[(c)] $f \in \VMO(\C)$.
\end{itemize}
In particular, $H_{f,(k)}$ is compact if and only if $H_{\bar f,(k)}$ is compact and $H_{f,n}$ is compact if and only if $H_{\bar f,n}$ is compact.
\end{thm}

\begin{proof}
Assume that $H_{f,(k)}$ is compact. Then, by Corollary \ref{cor:Hankel_compact}, $W_{-z}H_{f,(k)}W_z \to 0$ strongly as $\abs{z} \to \infty$. Let $x \in \beta\C \setminus \C$ and choose a net $(z_{\gamma})$ that converges to $x$. Then $W_{-z_{\gamma}}T_{f,(k)}W_{z_{\gamma}}$ converges strongly to $(T_{f,(k)})_x$. It follows that 
\[f(\cdot + z_{\gamma})g = W_{-z_{\gamma}}M_fW_{z_{\gamma}}g = W_{-z_{\gamma}}T_{f,(k)}W_{z_{\gamma}}g + W_{-z_{\gamma}}H_{f,(k)}W_{z_{\gamma}}g \to (T_{f,(k)})_xg\]
for all $g \in \Fc^2_{(k)}$ as $z_{\gamma} \to x$. In particular, for $g(w) = \bar{w}^{k-1}$, we get
\begin{equation} \label{eq:L^2_convergence}
\lim\limits_{z_{\gamma} \to x} \int_{\C} \abs{f(w+z_{\gamma}) - \psi(w)}^2 \abs{w}^{2k-2} \, \mathrm{d}\mu(w) = 0,
\end{equation}
where $\psi(w) := \frac{((T_{f,(k)})_xg)(w)}{\bar{w}^{k-1}}$ for $w \neq 0$. As $\norm{f(\cdot + z_{\gamma})}_{\infty} \leq \norm{f}_{\infty}$ for all $\gamma$, we also have $\norm{\psi}_{\infty} \leq \norm{f}_{\infty}$. Moreover, $g\psi \in \Fc^2_{(k)}$. By Liouville's theorem for polyanalytic functions, this implies that $g\psi$ is a polynomial in $w$ and $\bar w$ of degree at most $k-1$ (see \cite[Theorem 2.2]{Balk} or \cite[Corollary 1]{Krajkiewicz}). But the only true polyanalytic functions of order $k$ that are also a polynomial of degree $k-1$ are multiples of $g$. It follows that $\psi$ is constant. Furthermore, \eqref{eq:L^2_convergence} implies
\[\lim\limits_{z_{\gamma} \to x} \int_{\C} \abs{f(w+z_{\gamma}) - \psi}^2 \abs{w}^{2k-2} \abs{\phi(w)}^2 \, \mathrm{d}\mu(w) = 0\]
for all bounded functions $\phi$. This shows that the net of multiplication operators $M_{f(\cdot + z_{\gamma})}$ converges strongly to $\psi I$ on the set $\set{g\phi \in L^2(\C,\mu) : \phi \in L^{\infty}(\C,\mu)}$. This set is dense in $L^2(\C,\mu)$ and therefore $M_{f(\cdot + z_{\gamma})} \to \psi I$ strongly on $L^2(\C,\mu)$ as $z_{\gamma} \to x$. This of course implies $(T_{f,(k)})_x = \psi I$ and $(T_{\abs{f}^2,(k)})_x = \abs{\psi}^2 I$, but also $W_{-z_{\gamma}}H_{f,(k')}W_{z_{\gamma}} \to 0$ and $W_{-z_{\gamma}}H_{f,n}W_{z_{\gamma}} \to 0$ in the strong operator topology for any $k',n \in \N$. Corollary \ref{cor:Hankel_compact} therefore reveals that $H_{f,(k')}$ and $H_{f,n}$ are compact. The same argument as above, for $k = 1$, also shows that if $H_{f,n}$ is compact, then $W_{-z_{\gamma}}H_{f,(k')}W_{z_{\gamma}} \to 0$ for any $k' \in \N$. This proves the equivalence of (a) and (b), and that the compactness is independent of $k$ and $n$. For $n = 1$, the equivalence of (b) and (c) is well-known (see e.g.~\cite[Theorems 8.5 and 8.13]{Zhu}), but also follows directly from our previous results. Indeed, $M_{f(\cdot + z_{\gamma})} \to \psi I$ implies
\[\widetilde{\abs{f}^2}(x) - |\tilde{f}(x)|^2 = \abs{\psi}^2 - \abs{\psi}^2 = 0\]
for all $x \in \beta\C^n \setminus \C^n$ via Lemma \ref{lem:constant_limit_operators}, hence $f \in \VMO(\C)$.

Conversely, assume that $f \in \VMO(\C)$. Then the formula $H_{f,1}^*H_{f,1} = T_{|f|^2,1} - T_{\bar f,1}T_{f,1}$ implies $\Bc_1(H_{f,1}^*H_{f,1}) = \widetilde{\abs{f}^2} - \norm{T_{f,1}k_z}^2$. Clearly, $|\tilde{f}| \leq \norm{T_{f,1}k_z}$ and so
\[0 \leq \Bc_1(H_{f,1}^*H_{f,1}) \leq \widetilde{\abs{f}^2} - |\tilde{f}|^2.\]
It follows $\Bc_1(H_{f,1}^*H_{f,1}) \in C_0(\C)$. Theorem \ref{thm:compact_Berezin} now implies that $H_{f,1}^*H_{f,1}$ is compact.
\end{proof}

\section{Remarks and open problems} \label{sec:remarks}

It was quite surprising to the author that the compactness of the Hankel operators $H_{f,(k)}$ is independent of $k$. Indeed, one is tempted to define $\VMO_{(k)}$-spaces consisting of bounded functions $f$ satisfying
\[\Bc_{(k)}(\abs{f}^2) - |\Bc_{(k)}(f)|^2 \in C_0(\C).\]
The same argument as in the proof of Theorem \ref{thm:compact_Hankel} then shows that $H_{f,(k)}$ is compact if and only if $f \in \VMO_{(k)}$. But as it turns out, these $\VMO_{(k)}$-spaces are all the same. This naturally leads to the following question.

\begin{qn}
Is the compactness of Toeplitz operators $T_{f,(k)}$ also independent of $k$? The argument used in the proof of Theorem \ref{thm:compact_Toeplitz} does not quite work for $k \geq 2$, unfortunately.
\end{qn}

In this paper we quite heavily used the functions $\Bc_n$, which we introduced to generalize the Berezin transform. However, as the polyanalytic Fock spaces are reproducing kernel Hilbert spaces, it seems more natural to use the Berezin transform induced by the normalized reproducing kernels. This did not turn out to be very fruitful for our approach. Nevertheless, we pose the following question.

\begin{qn}
Does the more standard Berezin transform $z \mapsto \sp{Tk_{z,n}}{k_{z,n}}$ have any useful properties in connection with compactness or Fredholmness problems on $\Fc^2_n$? It is fairly obvious that if $T \in \Lc(\Fc^2_n)$ is compact, then $\sp{Tk_{z,n}}{k_{z,n}} \to 0$ as $\abs{z} \to \infty$, but we also know that this Berezin transform is not strong enough to characterize compactness in $\Ac^2_n$. Nevertheless, if restricted to Toeplitz operators (with certain symbols) maybe something can still be said.
\end{qn}

We have seen in Corollary \ref{cor:band_dominated} that every Toeplitz operator with bounded symbol is band-dominated. Bauer and Fulsche \cite{BaFu} showed that the algebra of band-dominated operators on $\Fc^2$ is generated by Toeplitz operators. A natural question is therefore:

\begin{qn}
Are $\Ac^2_{(k)}$ and $\Ac^2_n$ also generated by Toeplitz operators for $k,n \geq 2$? If not, can the band-dominated operators in Theorem \ref{thm:compact_Berezin} and related results at least be replaced by an algebra of Toeplitz operators?
\end{qn}

\end{document}